\RequirePackage{amsmath}
\documentclass{article}
\usepackage[latin1]{inputenc}
\usepackage{amsfonts}
\usepackage{amssymb}
\usepackage{enumitem}
\usepackage{geometry}
\usepackage{hyperref}
\hypersetup{colorlinks=true,urlcolor=blue,citecolor=blue}

\usepackage{amsthm}
\usepackage{thmtools}
\usepackage[capitalise]{cleveref}
\usepackage{xr}
\newtheorem{theorem}{Theorem}[section]
\newtheorem{proposition}[theorem]{Proposition}
\newtheorem{corollary}[theorem]{Corollary}
\newtheorem{definition}[theorem]{Definition}

\declaretheorem[style=remark,qed=$\Diamond$,Refname={Remark,Remarks},sibling=theorem]{remark}

\Crefname{assumption}{Assumption}{Assumptions}

\newcommand{\setto}{\rightrightarrows} 
\newcommand{\Fix}{\operatorname{Fix}}
\DeclareMathOperator{\StrFix}{\mathbf{Fix}}

\newcommand{\range}{\operatorname{range}}

\DeclareMathOperator*{\argmin}{arg\,min}
\DeclareMathOperator*{\Limsup}{Lim\,sup}

\author{Matthew K. Tam\thanks{Institut f\"ur Numerische und Angewandte Mathematik,
		Universit\"at G\"ottingen, 37083 G\"ottingen, Germany.
		Email:~\href{mailto:m.tam@math.uni-goettingen.de}{m.tam@math.uni-goettingen.de}} }
\title{Algorithms Based on Unions of Nonexpansive Maps}

\begin{document}
	
\maketitle

\begin{abstract}
	In this work, we consider a framework for the analysis of iterative algorithms which can be described in terms of a structured set-valued operator. More precisely, at each point in the ambient space, we assume that the value of the operator can be expressed as a finite union of values of single-valued paracontracting operators. Our main result, which shows that the associated fixed point iteration is locally convergent around strong fixed points, generalises a theorem due to Bauschke~\&~Noll~(2014).
\end{abstract}


                     
\section{Preliminaries}
Our setting throughout is the Euclidean space $X=\mathbb{R}^k$ equipped with norm $\|\cdot\|$. Following Elser \emph{et al} \cite{Elsner}, we define the class of \emph{paracontracting operators} as follows.
\begin{definition}[Paracontracting operators {\cite{Elsner}}]
 A continuous operator, $S:X\to X$, is said to be \emph{paracontracting} if 
  \begin{equation}\label{eq:qne inequality}
  \|S(x)-y\|<\|x-y\|\quad(\forall x\in X\setminus\Fix S)(\forall y\in \Fix S),
  \end{equation}
where $\Fix S=\{x\in X:S(x)=x\}$ denotes the fixed point set of the (single-valued) operator $S$.
\end{definition}
The class of paracontracting operators includes, in particular, all \emph{averaged nonexpansive} and \emph{firmly nonexpansive} operators; see \cite[\S4]{BauComb}. In the literature, a possibly discontinuous operators satisfying \eqref{eq:qne inequality} are called \emph{strictly quasinonexpansive}. Indeed, a paracontracting operator could be equally well described as a continuous, strictly quasinonexpansive operator.

In this work, we consider a structured set-valued operator, $T:X\setto X$, whose values can be expressed as the union of paracontracting operators taken from a finite collection of such operators. To be precise, let $I$ denote a finite index set. We consider operators $T:X\setto X$ of the form
\begin{equation}
  \label{eq:T first definition}
  T(x) := \left\{T_i(x):i\in\phi(x)\right\}\quad\forall x\in X,
\end{equation}  
where, for each $i\in I$, the single-valued operator $T_i:X\to X$ is paracontracting, and the \emph{active selector}, $\phi:X\setto I$, dictates which operators from the collection $\{T_i\}_{i\in I}$ are active and satisfies the following two properties.
\begin{enumerate}[label=(P\arabic*)]
	\item\label{P:nonempty} $\phi(x)$ is non-empty for every $x\in X$.
	\item\label{P:lsc} $\phi$ is \emph{outer semicontinuous (osc)} \cite[\S3.2]{DontRock}, that is, for all $x\in X$ it holds that
	   $$\phi(x)\supseteq\Limsup_{y\to x}\phi(y):=\{i\in I:\exists x_n\to x,\,\exists i_n\to i\text{ with }i_n\in\phi(x_n)\},$$
	where the limit superior is taken in the sense of \emph{Painlev\'e--Kuratowski} \cite[\S3.1]{DontRock}.
\end{enumerate}
Note that, for set-valued maps, there are two notions of fixed points which both coincide for single-valued operators; the \emph{fixed point set} denoted $\Fix T:=\{x:x\in T(x)\}$, and the \emph{strong fixed point set} denoted $\StrFix T:=\{x:T(x)=\{x\}\}$. It is clear that $\StrFix T\subseteq\Fix T$.

In what follows, we shall refer to operators of the form \eqref{eq:T first definition} which satisfy the aforementioned properties as \emph{union paracontracting}.

\begin{definition}[Union paracontracting]\label{def:union para}
  An  operator $T:X\setto X$ is said to be \emph{union paracontracting} if there exist a finite index set, $I$, a collection of single-valued paracontracting operators, $\{T_i\}_{i\in I}$, and an active selector, $\phi$, satisfying \ref{P:nonempty}--\ref{P:lsc} such that $T$ is of the form specified by \eqref{eq:T first definition}.
\end{definition}

In this work, we study convergence properties of the fixed-point iteration of union paracontracting operators. More precisely, given an initial point, $x_0\in X$, we consider a sequence $(x_n)_{n\in\mathbb{N}}$ generated by iterating the operator $T$ in the sense that
\begin{equation}
 \label{eq:T algorithm}
  x_{n+1} \in T(x_n)\quad\forall n\in\mathbb{N}.    
\end{equation} 
Equivalently, for each $n\in\mathbb{N}$, there is an index $i_n\in\phi(x_n)$ such that
\begin{equation}
 \label{eq:T algorithm realisation}
  x_{n+1}=T_{i_n}(x_n).
\end{equation}
The following result is a general convergence criterion for sequences conforming the structure of the iteration defined by \eqref{eq:T algorithm realisation}. 
\begin{theorem}[Convergence with admissible control {\cite[Theorem~1]{Elsner}}]\label{th:elsner}
   Let $\{T_j\}_{j\in J}$ be a collection of paracontracting operators on $X$ where the index set $J$ is assumed to be finite. Let $\{j_n\}_{n\in\mathbb{N}}\subseteq J$ be an \emph{admissible control} (\emph{i.e.,} each element in $J$ appears infinitely often in the sequence) and let $z_0\in X$ be given.  Then the sequence defined by
     $$ z_{n+1}=T_{j_n}(z_n)\quad\forall n\in\mathbb{N}, $$
   converges if and only if the collection of operators, $\{T_{j}\}_{j\in J}$, have a common fixed point. Moreover, in this case $\displaystyle\lim_{n\to\infty}z_n=\overline{z}\in\cap_{j\in J}\Fix T_{j}$.
\end{theorem}
Given the sequence $(x_n)_{n\in\mathbb{N}}$ in \eqref{eq:T algorithm realisation} defined by applying the operators $(T_{i_n})_{n\in\mathbb{N}}$, the \emph{admissible set} for $(x_n)_{n\in\mathbb{N}}$ is the (finite) set $I^*\subseteq I$ such that $i^*\in I^*$ appears infinitely often in $(i_n)_{n\in\mathbb{N}}$. Note that $I^*$ is always nonempty by virtue of the finiteness of $I$ and, moreover, the definition of $I^*$ ensures the existence of an $n_0\in\mathbb{N}$ such that $(i_n)_{n=n_0}^\infty\subseteq I^*$ is an admissible control.

\section{The Main Result: Local convergence}
In this section, we prove an abstract result concerning local convergence of iterations based on union contracting operators. We begin with the following proposition which applies, in particular, to active selectors. In what follows, $\mathbb{B}(x^*;\delta)$ denotes the \emph{closed ball} centered at $x^*$ with radius $\delta>0$.
\begin{proposition}\label{prop:basin of attraction}
  Suppose $\varphi:X\setto M$ for a finite set $M$, and let $x^*\in X$. Then $\varphi$ is outer semicontinuous at $x^*$ if and only if there exists $\delta>0$ such that $$\varphi(x)\subseteq\varphi(x^*)\quad\forall x\in\mathbb{B}(x^*;\delta).$$
\end{proposition}
\begin{proof}
  ($\Longrightarrow$):~Suppose, by way of a contradiction, that $\varphi$ is osc at $x^*$ but there exists a sequence $(x_n)_{n\in\mathbb{N}}$ with $x_n\in\mathbb{B}(x^*;1/n)$ such that $\varphi(x_n)\not\subseteq\varphi(x^*)$. Then there is a sequence, $(m_n)_{n\in\mathbb{N}}\subseteq M$, such that
  \begin{equation}
    \label{eq:m_n definition}
    m_n\in\varphi(x_n)\setminus\varphi(x^*) \quad\forall n\in\mathbb{N}.
  \end{equation} 
  Since $M$ is finite, by applying the pigeonhole principle, there exists a subsequence $(m_{k_n})_{n\in\mathbb{N}}$ such that $m_{k_n}=m\in M$ for all $n\geq n_0$. Altogether, we deduce that
   \begin{equation}
    \label{eq:m inclusion}
   	 m\in\Limsup_{x\to x^*}\varphi(x)\subseteq \varphi(x^*),
   \end{equation}
  where the last inclusion follows from the definition of osc. But then \eqref{eq:m inclusion} contradicts \eqref{eq:m_n definition}, and thus completes the first half of the proof.
  
  ($\Longleftarrow$):~Consider two arbitrary sequences, $(x_n)_{n\in\mathbb{N}}$ and $(i_n)_{n\in\mathbb{N}}$, such that $x_n\to x^*$ and $i_n\to i$ with $i_n\in\varphi(x_n)$. In order to deduce that $\varphi$ is osc at $x^*$, we must show that $i\in\varphi(x^*)$. To this end, since $M$ is finite, there exists an index $n_0\in\mathbb{N}$ such that $i_n=i$ for all $n\geq n_0$. By assumption, there is a $\delta>0$ such that $\varphi(x)\subseteq\varphi(x^*)$ for all $x\in\mathbb{B}(x^*;\delta)$ and, consequently, there is an index $n_1\geq n_0$ such that $\varphi(x_n)\subseteq\varphi(x^*)$ for all $n\geq n_1$. Altogether, 
    $$i=i_n\in\varphi(x_n)\subseteq\varphi(x^*)\quad\forall n\geq n_1.$$
  Thus we have that $i\in\varphi(x^*)$ and so the proof is complete. 
\end{proof}

We are now ready to prove our main result concerning local convergence of the algorithm defined by \eqref{eq:T algorithm}. In what follows, we adopt the convention that $\sup\emptyset=-\infty$. 
\begin{theorem}[Local convergence around strong fixed points]\label{th:local convergence}
  Suppose $T:X\setto X$ is union paracontracting with $x^*\in\StrFix T$, and define
   \begin{equation}
   \label{eq:attraction radius}
     r:=\sup\left\{\delta>0:\phi(x)\subseteq\phi(x^*)\text{ for all }x\in\mathbb{B}(x^*;\delta)\right\}.
   \end{equation}
  Then $r>0$ and, for any $\epsilon\in(0,r)$, it holds that $\|y-x^*\|\leq\|x-x^*\|$ whenever $x\in\mathbb{B}(x^*;\epsilon)$ and $y\in T(x)$. Furthermore, if the initial point $x_0$ is contained in $\mathbb{B}(x^*;\epsilon)$ and $x_{n+1}\in T(x_n)$ for all $n\in\mathbb{N}$, then the sequence $(x_n)_{n\in\mathbb{N}}$ converges to a point $\overline{x}\in\Fix T\cap\mathbb{B}(x^*;\epsilon)$.
\end{theorem}	
\begin{proof}
  Since $x^*\in\StrFix T$, it follows from Definition~\ref{def:union para} that $\phi(x^*)\subseteq I$ is nonempty with
    $$ x^* \in \bigcap_{i\in\phi(x^*)}\Fix T_i \neq\emptyset. $$
  As $\phi$ is osc at $x^*$ (Property~\ref{P:lsc}), Proposition~\ref{prop:basin of attraction} implies the existence of a $\delta>0$ such that $\phi(x)\subseteq\phi(x^*)$ for all $x\in\mathbb{B}(x^*;\delta)$. Consequently, the constant $r$, as defined in \eqref{eq:attraction radius}, is not less than $\delta$ and thus, in particular, strictly greater than zero. 
  
  Let $\epsilon\in(0,r)$. Since $\{T_i\}_{i\in I}$ is a collection of paracontracting operators, we have that $ \|T_i(x)-x^*\| \leq \|x-x^*\|$ whenever $i\in\phi(x^*)$ which implies that 
    $$ \|y-x^*\| \leq \|x-x^*\|,$$
  whenever $x\in\mathbb{B}(x^*;\epsilon)$ and $y\in T(x)$. This completes the proof of the first two assertions.
  
  Furthermore, suppose $x_0\in\mathbb{B}(x^*;\epsilon)$ and $x_{n+1}=T_{i_n}(x_n)\in T(x_n)$ with ${i_n\in\varphi(x_n)}$ for all $n\in\mathbb{N}$. Then it follows that $x_n\in\mathbb{B}(x^*;\epsilon)$ for all $n\in\mathbb{N}$ and, thus together with \ref{P:nonempty}, that the set of admissible indices for $(x_n)_{n\in\mathbb{N}}$, denoted $I^*$, is a nonempty subset of $\phi(x^*)$. Consequently,
    \begin{equation}\label{eq:FPS I*}
    x^*\in \bigcap_{i\in\phi(x^*)}\Fix T_i\subseteq \bigcap_{i\in I^*}\Fix T_i
    \end{equation}
  and, by the definition of $I^*$, there exists an $n_0\in\mathbb{N}$ such that $(i_n)_{n=n_0}^\infty\subseteq I^*$ is an admissible control. Noting \eqref{eq:FPS I*} and applying Theorem~\ref{th:elsner} (with $z_0=x_{n_0}$ and $J=I^*$) gives that $x_n\to\overline{x}\in \Fix T_{i^*}$ for some $i^*\in I^*$. The final claim regarding the limit, $\overline{x}$, follows by noting that $\Fix T_{i^*}\subseteq\Fix T$ and that $\mathbb{B}(x^*;\epsilon)$ is closed. 
\end{proof}	
	
\begin{remark}[Necessity of strong fixed points]
  In general, the assumption that $x^*$ is strong fixed point of $T$ cannot be relaxed to a mere fixed point. For an example, the reader is referred to \cite[Remark~2]{BauNoll}.
\end{remark}	

\begin{remark}[Limit points need only be weak fixed points]
   Although the reference point $x^*$ is assumed to be a strong fixed point, in general the limit point $\overline{x}$ of the sequence $(x_n)_{n\in\mathbb{N}}$ need only be a weak fixed point. An example of this behaviour follows.
   
   Consider the union paracontracting operator $T:\mathbb{R}\setto\mathbb{R}$ defined by
      $$T(x):=\{T_1(x),T_2(x)\}$$
   where $T_1(x)=0$ and $T_2(x)=x$ (\emph{i.e.,}~the active selector is given by $\varphi(x)=\{1,2\}$ for all $x\in\mathbb{R}$). Then $\StrFix T=\{0\}$ and $\Fix T=\mathbb{R}$ and, for any initial point $x_0\in\mathbb{R}\setminus\{0\}$, the sequence $(x_n)_{n\in\mathbb{N}}$ with $x_n=x_0$ for all $n\in\mathbb{N}$ satisfies $$x_{n+1}=T_2(x_n)\in T(x_n)\quad\forall n\in\mathbb{N}.$$
   This constant sequence converges to $\overline{x}=x_0\in\mathbb{R}\setminus\{0\}=\Fix T\setminus\StrFix T$. 
\end{remark}

\begin{remark}[The Douglas--Rachford algorithm]
  Theorem~\ref{th:local convergence} is a generalisation of \cite[Theorem~1]{BauNoll} which considers the case in which $T$ is a specific operator; the \emph{Douglas--Rachford operator} for two sets, $A$ and $B$, which can be realised as finite unions of closed convex sets $\{A\}_{i\in I}$ and $\{B_j\}_{j\in J}$. That is, $A:=\bigcup_{i\in I}A_i$, $B:=\bigcup_{j\in J}B_j$ and
    $$T=\frac{I+R_B\circ R_A}{2},$$
  where $P_A(x):=\argmin_{a\in A}\|x-a\|$ denotes the \emph{projector} onto $A$ and $R_A:=2P_A-I$.   In this case, the function $\phi:X\setto I\times J$ can be defined as
  $$\phi(x) := \{(i,j)\in I\times J:P_{A_i}(x)\in P_A(x),\,P_{B_j}(R_{A_i}(x))\in P_{B}(R_{A_i}(x))\}, $$
  and we therefore have that $T(x)=\{T_{ij}(x):i\in\phi(x)\}$ where $T_{ij}:X\to X$ is defined by
    $$ T_{ij}=\frac{I+R_{B_j}\circ R_{A_i}}{2} \quad\forall (i,j)\in I\times J, $$
  is firmly nonexpansive by \cite[Propositions~4.8~\&~4.21]{BauComb} and hence, in particular, paracontracting. It is straightforward to see that $\phi$ satisfies \ref{P:nonempty}. By combining Proposition~\ref{prop:basin of attraction} and \cite[Theorem~1(2)]{BauNoll}, we deduce that $\phi$ also satisfies \ref{P:lsc}. We therefore recover \cite[Theorem~1]{BauNoll} within our framework of Theorem~\ref{th:local convergence}.
\end{remark}
  
A consequence of Theorem~\ref{th:local convergence} is the following result concerning global convergence if there exists a strong fixed point such that every operator is active.
\begin{corollary}[Global convergence]\label{cor:global convergence}
 Suppose $T:X\setto X$ is union paracontracting and there exists a strong fixed point $x^*\in\StrFix T$ such that $\phi(x^*)=I$. For any $x_0\in X$, select $x_{n+1}\in T(x_n)$ for all $n\in\mathbb{N}$. Then the sequence $(x_n)_{n\in\mathbb{N}}$ converges to a point $\overline{x}\in\Fix T$.		
\end{corollary}  
\begin{proof}
  Since $\phi(x^*)=I$ and $\range(\phi)\subseteq I$, the constant $r>0$ in \eqref{eq:attraction radius} becomes
  $$r=\sup\left\{\delta>0:\phi(x)\subseteq I\text{ for all }x\in\mathbb{B}(x^*;\delta)\right\}=+\infty.$$
  The result then follows from Theorem~\ref{th:local convergence}. 
\end{proof}

\section{An Application: Sparsity constrained minimisation}  
In this section, we give an application of Theorem~\ref{th:local convergence} in \emph{sparsity constrained minimisation}. In particular, we consider minimisation problems (with $X=\mathbb{R}^k$) of the form
\begin{equation}
\label{eq:sparsity constraint minimisation}
\min_{x\in X}\{f(x):\|x\|_0\leq s\},
\end{equation}   
where $f:X\to\mathbb{R}$ is convex and continuously differentiable with Lipschitz continuous gradient, $\nabla f$, having Lipschitz constant $L>0$, $\|x\|_0$ denotes the \emph{$\ell_0$-functional} which counts the number of non-zero entries in a vector $x$, and $s\in\{0,1,\dots,k-1\}$ is an \emph{a priori} estimate on the sparsity of $x$ (typically $s\ll k$).

Denote $\mathcal{A}:=\{x\in X:\|x\|_0\leq s\}$. A \emph{forward-backward algorithm} (or a \emph{projected gradient algorithm}) \cite[(27.26)]{BauComb} for \eqref{eq:sparsity constraint minimisation} can now be compactly described as
  \begin{equation}\label{eq:fb T}
    x_{n+1}\in T(x_n):= P_{\mathcal{A}}\left(x_n-\gamma\nabla f(x_n)\right)\quad\forall n\in\mathbb{N},
  \end{equation}                  
where the constant $\gamma$ is selected such that $\gamma\in(0,2/L)$.  This algorithm also appears, for instance, in \cite{YuanLiZhang} under the name \emph{fast gradient hard thresholding pursuit (FGraHTP)}.
  
In order to investigate the properties of the forward-backward operator for \eqref{eq:sparsity constraint minimisation}, we introduce some further notation. Denote $\mathcal{I}:=\{I\in 2^{\{1,2,\dots,k\}} : |I|=s\}$. Then the set $\mathcal{A}$ can be expressed as the union of nonempty subspaces, precisely
$$\mathcal{A} = \bigcup_{I\in\mathcal{I}} A_I\text{~~where~~}A_I:=\{x\in X:x_i\neq 0\text{ only if }i\in I\}.$$
Consequently its projector has the form
$$ P_{\mathcal{A}}(x)=\{P_{A_I}(x):I\in\phi(x)\},$$
where $\phi(x):= \left\{I\in\mathcal{I}:\|x-P_{A_I}(x)\| = \min_{J\in\mathcal{I}}\|x-P_{A_{J}}(x)\|\right\}$ 
and, for any $I\in\mathcal{I}$, the single-valued projector onto $A_I$ is given componentwise by
	\begin{equation}
	\label{eq:PAI}
	 P_{A_I}(x)_i = \begin{cases}
	 x_i & i\in I\\
	 0   & i\not\in I,\\
	 \end{cases}\qquad\forall i\in\{1,2,\dots,k\}.
	\end{equation}
 The operator $T$ may then be expressed in terms of single-valued operators as
\begin{equation}\label{eq:T representation}
	T(x)=\{(P_{A_I}\circ Q)(x):I\in\varphi(x)\},
\end{equation}
where we denote $Q:=I-\gamma\nabla f$ and $\varphi:=(\phi\circ Q)$. Note that by \cite[Proposition~3.6]{restricted} we see that $\varphi(x)$ can be equivalently described as
\begin{equation}\label{eq:varphi equivlanet}
\varphi(x) = \left\{J\in\mathcal{I}: \min_{i\in J}|Q(x)_i| \geq \max_{i\not\in J}|Q(x)_i|\right\}.
\end{equation}

\begin{proposition}[Forward-backward with a sparsity constraint]
	The forward-backward operator for \eqref{eq:sparsity constraint minimisation}, $T:X\setto\mathcal{A}$,  defined by \eqref{eq:fb T} is union paracontracting.
\end{proposition}  
\begin{proof}
    We use the representation \eqref{eq:T representation} to show that $T$ is union paracontracting.
    First note that, for each $J\in\mathcal{I}$, the argument in \cite[Theorem~25.8]{BauComb} shows that the (single-valued) operator $P_{A_J}\circ Q$ is averaged nonexpansive and thus, in particular, also paracontracting. The finiteness of the index set $\mathcal{I}$ together with the definition of $\phi$ imply that $\varphi$ is nonempty valued. To see that $\varphi$ is osc, consider two sequences, $(x_n)_{n\in\mathbb{N}}$ and $(I_n)_{n\in\mathbb{N}}$, such that $x_n\to x\in X$ and $I_n\to I\in\mathcal{I}$ such that $I_n\in\varphi(x_n)$. Since $\mathcal{I}$ is finite, it follows that $I=I_n$ for all sufficiently large $n$ and hence, using the fact that both $P_{A_J}$ and $Q$ are continuous, we deduce that
    \begin{align*}
      \|Q(x)-P_{A_I}(Q(x))\|  
      &= \lim_{n\to\infty}\|Q(x_n)-P_{A_I}(Q(x_n))\| \\
      &= \lim_{n\to\infty}\left(\min_{J\in \mathcal{I}}\|Q(x_n)-P_{A_J}(Q(x_n))\|\right) \\
      & = \min_{J\in \mathcal{I}}\left(\lim_{n\to\infty}\|Q(x_n)-P_{A_J}(Q(x_n))\|\right) \\
      &= \min_{J\in \mathcal{I}}\|Q(x)-P_{A_J}(Q(x))\|,
    \end{align*}
    which shows that $I\in\varphi(x)$ and we conclude that $\varphi$ is osc. This shows that $T$ is union paracontracting and completes the proof. 
\end{proof}

For a vector $x\in X$, let $\sigma_t(x)$ denote its $t$-th largest entry in magnitude, so that
$$\sigma_1(x)\geq \sigma_2(x)\geq \dots\geq\sigma_k(x)\geq0.$$
The following describes the fixed point sets of the forward-backward operator. Given a subspace $S\subseteq X$, its  \emph{orthogonal complement}, denoted $S^\perp$, is defined as $S^\perp := \{x\in X:\langle x,s\rangle =0,\,\forall s\in X\}$.
\begin{proposition}[Fixed points and local minima]\label{prop:fixed points local minima}
	Let $T:X\setto\mathcal{A}$ denote the forward-backward operator for \eqref{eq:sparsity constraint minimisation} and $x\in X$. The following assertions hold.
	\begin{enumerate}[label=(\roman*)]
		\item\label{it:fp i} Let $I\in\mathcal{I}$. Then $\Fix(P_{A_I}\circ Q)=\argmin_{z\in A_I}f(z)$. 
		\item\label{it:fp ii} Let $I\in\mathcal{I}$. If $x\in \Fix(P_{A_I}\circ Q)$ then $x\in A_I$ and $\nabla f(x)\in A_I^\perp.$
		\item\label{it:fp iii} $x\in\Fix T$ if and only if there exists an $I\in\varphi(x)$ such that $x\in\argmin_{z\in A_I}f(z)$. Moreover, in this case, $x$ satisfies
		\begin{equation}\label{eq:phi inequality}
		\min_{i\in I}|x_i|\geq \gamma\|\nabla f(x)\|_\infty. 
		\end{equation}
		\item\label{it:fp iv} If $x\in\Fix T$ and $ \|x\|_0=s$ (resp.\ $\|x\|_0<s$) then $x$ is a local (resp.\ global) minimum of \eqref{eq:sparsity constraint minimisation}. In particular, every weak fixed point of $T$ is a local minimum of \eqref{eq:sparsity constraint minimisation}.
        \item\label{it:fp v} $x\in\StrFix T$ if and only if $x\in\Fix T$ and $\cup_{I\in\varphi(x)}\argmin_{z\in A_I}f(z)$ is a singleton.  In particular, if $x\in\Fix T$ then $x\in\StrFix T$ as soon as either
  \begin{equation}\label{eq:strong fixed point condition}
  \nabla f(x)=0\quad\text{or}\quad
  \sigma_s(x)>\gamma\|\nabla f(x)\|_\infty .
  \end{equation}    
	\end{enumerate}		
\end{proposition}	
\begin{proof}
	\ref{it:fp i}:~Follows from \cite[Corollary~26.5(i)\&(viii)]{BauComb}.
	\ref{it:fp ii}:~Since $A_I$ is a subspace its projector, $P_{A_I}$, is a linear operator. As $x\in A_I$, we therefore have
	$$ x = P_{A_I}(x-\gamma\nabla f(x)) = P_{A_I}(x)-\gamma P_{A_I}(\nabla f(x)) = x- \gamma P_{A_I}(\nabla f(x)),$$
	which implies that $P_{A_I}(\nabla f(x))=0$. The claim follows.
	\ref{it:fp iii}:~By the definition of $T$ and \ref{it:fp i}, it follows that $x\in\Fix T$ if and only if there exists an $I\in\varphi(x)$ such that $x\in\Fix(P_{A_I}\circ Q)=\argmin_{z\in A_I}f(z)$. To prove \eqref{eq:phi inequality}, combine \ref{it:fp ii}, \eqref{eq:PAI} and the fact that $I\in\varphi(x)$ to deduce
	$$ \min_{i\in I}|x_i| = \min_{i\in I}|x_i-\gamma\nabla f(x)_i| \geq \max_{i\not\in I}|x_i-\gamma\nabla f(x)_i| =  \max_{i\not\in I}|\gamma\nabla f(x)_i|. $$
	\ref{it:fp iv}:~Suppose $x\in\Fix T$. By \ref{it:fp iii}, there exists an $I\in\varphi(x)$ such that 
	\begin{equation}
	\label{eq:fp inequality}
	x\in\argmin_{z\in A_I}f(z)\text{~~and~~}\min_{i\in I}|x_i|\geq \gamma\|\nabla f(x)\|_\infty.
	\end{equation}
	On one hand, if $\|x\|_0<s$ then $\min_{i\in I}|x_i|=0$ and \eqref{eq:fp inequality} implies ${\nabla f(x)=0}$. Since $f$ is convex and continuous, it follows that $x$ is a global minimum to the unconstrained problem $\min_{z\in X}f(z)$ (see \cite[Proposition~26.1]{BauComb}). Since $x\in\mathcal{A}$, it must also be a global minimum of the constrained problem $\min_{z\in\mathcal{A}}f(z)$.	
	On the other hand, if $\|x\|_0=s$, denote $\delta:=\min_{i\in I}|x_i|$. Then \cite[Lemma~3.4]{restricted} implies the inclusion
	$$ \{y\in\mathcal{A}:\|x-y\|\leq\delta\} \subseteq A_I. $$
	In other words, there is an open ball centered at $x$ whose intersection with $\mathcal{A}$ is contained entirely in $A_I$. Together with \eqref{eq:fp inequality}, this shows that $x$ is a local minimum of $\min_{z\in\mathcal{A}}f(z)$. 
	\ref{it:fp v}:~The equivalence follows from \ref{it:fp i} and the definition of the strong fixed point set. In particular, let $x\in\Fix T\subseteq\mathcal{A}$. On one hand, if $\nabla f(x)=0$, then  
	$T(x)=P_{\mathcal{A}}(x-\gamma\cdot 0)=P_{\mathcal{A}}(x)=\{x\},$
	and hence $x\in\StrFix T$. On the other hand, the definition of $\Fix T$ ensures the existence of an index set $I\in\varphi(x)$ such that $x\in\Fix(P_{A_I}\circ Q)$. \ref{it:fp ii} then gives $\nabla f(x)\in A_I^\perp$ which, by combining with \eqref{eq:strong fixed point condition}, yields the inequality
	\begin{equation}
	\label{eq:varphi inequality}
	\min_{i\in I}|Q(x)_i|=\min_{i\in I}|x_i| = \sigma_s(x) > \gamma\|\nabla f(x)\|_\infty = \max_{i\not\in I}|\gamma\nabla f(x)_i|=\max_{i\not\in I}|Q(x)_i|.
	\end{equation}                   
	Since the inequality in \eqref{eq:varphi inequality} is strict, \eqref{eq:varphi equivlanet} together with the definition of $\mathcal{I}$ shows that $\{I\}=\varphi(x)$. It then follows that $x\in\StrFix T$ which completes the proof. 
\end{proof}

Altogether, we arrive at the following specialisation of Theorem~\ref{th:local convergence} for \eqref{eq:sparsity constraint minimisation}.

\begin{corollary}[Local convergence for sparsity constrained minimisation]\label{cor:FGraHTP}
  The conclusions of Theorem~\ref{th:local convergence} apply to the forward-backward operator for \eqref{eq:sparsity constraint minimisation} defined by \eqref{eq:fb T}. In particular, around strong fixed points of $T$, the iteration \eqref{eq:fb T} is locally convergent to a local minimiser of \eqref{eq:sparsity constraint minimisation}.
\end{corollary}

\begin{remark}[Restricted isometry/convexity/smoothness properties]
	Many results in the literature concerning the behaviour of the forward-backward algorithm applied to \eqref{eq:sparsity constraint minimisation} proceed by assuming a strong hypothesis such as  \emph{restricted isometry property} or one of its variants. For instance, \cite[Theorems~1~\&~2]{YuanLiZhang} rely on \emph{restricted strong convexity/smoothness} of the function $f$. Such properties typically provide a guarantee on accuracy of the recovered solutions and not only convergence. Whilst Corollary~\ref{cor:FGraHTP} provides no such guarantee, we emphasis that it also does not require any ``restricted" property be assumed, and so is complementary to any such result.
\end{remark}

\paragraph{Acknowledgements.}
 This work was supported by a Postdoctoral Fellowship from the Alexander von Humboldt Foundation. The author wishes to thank Yura Malitsky for suggesting Corollary~\ref{cor:global convergence} and the two anonymous referees whose comments have improved the manuscript.

\end{document}